\documentclass[12pt,a4]{amsart}
\usepackage{a4wide}
\usepackage{amsmath,amscd,amssymb}
\usepackage{amsfonts}
\usepackage{amsthm}
\usepackage{color}
\usepackage{enumitem}
\usepackage{xfrac}
\usepackage{hyperref}

\theoremstyle{plain}
\newtheorem{mainthm}{Theorem}

\newtheorem*{Thm*}{Theorem}
\newtheorem*{Claim*}{Claim}
\newtheorem{Thm}{Theorem}[section]
\newtheorem{Cor}[Thm]{Corollary}
\newtheorem{Lem}[Thm]{Lemma}
\newtheorem{Prop}[Thm]{Proposition}
\theoremstyle{definition}		
\newtheorem{Def}[Thm]{Definition}
\newtheorem{Rem}[Thm]{Remark}
\newtheorem{Ex}[Thm]{Example}

\renewcommand{\l}{\mbox{${\langle}$}}
\renewcommand{\r}{\mbox{${\rangle}$}}

\newcommand{\es}{\emptyset}
\DeclareMathOperator{\im}{im}

\DeclareMathOperator{\End}{End}

\renewcommand{\a}{\alpha}
\renewcommand{\b}{\beta}
\newcommand{\g}{\gamma}
\renewcommand{\d}{\delta}
\newcommand{\e}{\varepsilon}
\newcommand{\lam}{\lambda}

\renewcommand{\th}{\theta}
\renewcommand{\phi}{\varphi}

\newcommand{\w}{\omega}

\renewcommand{\phi}{\varphi}

\renewcommand{\bf}{\mathbf}

\newcommand{\N}{\mathbb{N}}
\newcommand{\Z}{\mathbb{Z}}
\newcommand{\Q}{\mathbb{Q}}
\newcommand{\R}{\mathbb{R}}

\newcommand{\AND}{\qquad\text{and}\qquad}

\newenvironment{thmenumerate}{
   \begin{enumerate}[label=\textup{(\roman*)}, widest=(5), leftmargin=10mm]}{
    \end{enumerate}}

\newcommand{\bit}{\begin{itemize}[label=\textbullet, leftmargin=5mm]}
\newcommand{\eit}{\end{itemize}}
\newcommand{\ben}{\begin{thmenumerate}}
\newcommand{\bena}{\begin{enumerate}[label=\textup{(\alph*)},leftmargin=10mm]}
\newcommand{\een}{\end{thmenumerate}}
\newcommand{\eena}{\end{enumerate}}

\begin{document} 
\title{~\\[-10mm]Diameters of endomorphism monoids of chains}
\author{James East, Victoria Gould, Craig Miller, Thomas Quinn-Gregson}
\address{Centre for Reseach in Mathematics and Data Science, Western Sydney University, Australia}
\email{J.East@WesternSydney.edu.au}
\address{Department of Mathematics, University of York, UK, YO10 5DD}
\email{victoria.gould@york.ac.uk, craig.miller@york.ac.uk}
\address{~\vspace{-1em}}
\email{tquinngregson@gmail.com}
\thanks{This work was supported by the grants: FT190100632 of the ARC and EP/V002953/1 of the EPSRC}
\maketitle

\vspace{-1.5em}
\begin{abstract}
The left and right diameters of a monoid are topological invariants defined in terms of suprema of lengths of derivation sequences with respect to finite generating sets for the universal left or right congruences.  We compute these parameters for the endomorphism monoid $\End(C)$ of a chain $C$.  Specifically, if $C$ is infinite then the left diameter of $\End(C)$ is 2, while the right diameter is either 2 or 3, with the latter equal to 2 precisely when $C$ is a quotient of $C{\setminus}\{z\}$ for some endpoint $z$.  If $C$ is finite then so is $\End(C),$ in which case the left and right diameters are 1 (if $C$ is non-trivial) or 0.
\end{abstract}

\vspace{0.5em}

\section{Introduction}\label{sec:intro}

For a monoid $S$ whose universal right congruence is generated by a finite set $U,$ the right diameter of $S$ with respect to $U$ is the supremum of the minimum lengths of derivations for pairs $(a,b)\in S\times S$ as a consequence of those in $U.$  The right diameter of $S$ is then the minimum of the set of all right diameters with respect to finite generating sets.  Thus, a monoid has finite right diameter if and only if its universal right congruence is finitely generated and there is a bound on the length of sequences required to relate any two elements as a consequence of the generators.  More precise definitions regarding right diameter will be given in Section \ref{sec:diameter}.  The notion of left diameter is dual.

Monoids with finite (left) diameter were introduced by White, under the name `left pseudo-finite', in the context of Banach algebras \cite{White:2017}.  This work was motivated by a conjecture of Dales and {\.Z}elazko, stating that a unital Banach algebra in which every maximal left ideal is finitely generated is necessarily finite dimensional.  

Monoids with finite left diameter were first studied systematically in \cite{Dandan} (again under the name `left pseudo-finite'), within the broader context of monoids having a finitely generated universal left congruence.  The latter condition was shown to be equivalent to a number of previously-studied properties, including the homological finiteness property of being type left-$FP_1$ and connectivity of certain Cayley graphs \cite[Theorem 3.10]{Dandan}\cite{Kob}.

The article \cite{Gould:23} investigated the existence and nature of minimal ideals in monoids with finite diameter.  It was observed that the notion of having right/left diameter 1 can be reformulated in terms of the so-called diagonal acts.  For a monoid $S,$ the diagonal right $S$-act is the set $S\times S$ under the right action given by $(a,b)c=(ac,bc).$  A monoid has right diameter 1 if and only if its diagonal right act is finitely generated.  The latter property had already been studied by Gallagher and Ru\v{s}kuc \cite{Gallagher:2006,Gallagher:2005,Gallagher}.  Some of their most interesting results concern certain natural infinite monoids of transformations and relations.  In particular, for any infinite set $X,$ the diagonal left and right acts are monogenic (i.e.\ generated by a single pair) for the monoids of all binary relations on $X,$ of all partial transformations on $X,$ and of all (full) transformations on $X$ \cite{Gallagher:2005}.

The above work regarding transformation monoids was part of the motivation for the article \cite{East}, which determines the left and right diameters of a large class of natural monoids of transformations and partitions.  For instance, it is shown that the monoid of all injective mappings on $X$ has right diameter 4, but its universal left congruence is not even finitely generated, and the monoid of all surjective mappings on $X$ has left diameter 4, while its universal right congruence is not finitely generated.  It is intriguing that all the diameters computed in \cite{East} are no greater than 4.

Some of the most natural and well-studied monoids arise as endomorphism monoids of mathematical structures, be they relational or algebraic.  Conversely, one way to classify complicated structures is via properties of their endomorphism monoids.  Chains, also known as linearly/totally ordered sets, are important and well-known relational structures; see \cite{Davey} for instance.  Of course, a chain $C$ may also be regarded as an algebraic structure, specifically a semigroup with the operation that takes the minimum (or, alternatively, the maximum) of two elements, and the order-preserving transformations of $C$ are precisely the semigroup endomorphisms of $C.$  We denote the endomorphism monoid of a chain $C$ by $\End(C)$\footnote{In the literature, $\End(C)$ is sometimes denoted by $\mathcal{O}_C$ or $\mathcal{OT}\!_C$.}.  Such monoids have received significant attention, particularly in the case of finite chains.  For instance, for a finite chain $C,$ a presentation for $\End(C)$ is provided in \cite{Aizen}, combinatorial properties of $\End(C)$ are investigated in \cite{Laradji}, the endomorphisms of $\End(C)$ are characterised in \cite{Fernandes}, and the representation theory of $\End(C)$ is considered in \cite{Stein}.
Endomorphism monoids of infinite chains can possess a rather complex structure, even in the case of the chain $\Q$ of rational numbers (under the natural ordering) \cite{McPhee}.  A particularly fruitful line of investigation into the behaviour of such monoids concerns the notion of regularity.  A map $\a\in\End(C)$ is regular if there exists $\b\in\End(C)$ such that $\a=\a\b\a,$ and $\End(C)$ is regular if all its elements are.  For a survey of regularity in $\End(C),$ see the article \cite{Mora}, which characterises the regular mappings of $\End(C).$  Notably, the monoid $\End(C)$ is regular if $C$ is isomorphic to a subchain of $\Z$ \cite{Kemp,Kim}, and hence if $C$ is finite.  On the other hand, the monoids $\End(\Q)$ and $\End(\R)$ are not regular \cite{Rung}.  We note that certain regular mappings will play a key role in the present paper.

In the context of diameter, Gallagher has shown that for any infinite chain $C,$ the diagonal left and right acts of $\End(C)$ are not finitely generated \cite[Theorems 4.5.4 and 4.5.5]{Gallagher}, and hence $\End(C)$ has neither left nor right diameter 1.  
Our main results provide complete classifications of the left and right diameters of endomorphism monoids of (infinite) chains:

\begin{mainthm}\label{thm:A}
For any infinite chain $C,$ the left diameter of $\End(C)$ is 2.
\end{mainthm}

\begin{mainthm}\label{thm:B}
For any infinite chain $C,$ the right diameter of $\End(C)$ is either 2 or 3.  Furthermore, $\End(C)$ has right diameter 2 if and only if $C$ has an endpoint $z$ such that $C$ is a quotient of $C{\setminus}\{z\}.$
\end{mainthm}

The paper is structured as follows.  In Section \ref{sec:diameter} we formally define the notion of diameter for monoids, and state some results we need from the literature.  In Section \ref{sec:chains} we establish the necessary preliminary material on chains and their endomorphisms.  Section \ref{sec:es} is concerned with the crucial condition of Theorem \ref{thm:B} (which determines the right diameter of the endomorphism monoid).  Finally, in Section \ref{sec:main} we prove Theorems \ref{thm:A} and \ref{thm:B}.  We note that we will be working in standard ZFC set theory; see \cite{Jech} for instance.

\section{Diameters of Monoids}\label{sec:diameter}

Throughout this section, $S$ denotes a monoid.  We denote the set of positive integers by $\N,$ and let $\N_0=\N\cup\{0\}.$


An equivalence relation $\rho$ on $S$ is a {\em right congruence} if $(a,b)\in\rho$ implies $(as,bs)\in\rho$ for all $s\in S$. 
For $U\subseteq S\times S,$ the {\em right congruence generated by} $U$ is the smallest right congruence on $S$ containing $U$; we denote this right congruence by $\l U\r.$

\begin{Lem}\label{lem:sequence}\cite[Lemma I.\,4.\,37]{kkm}
Let $S$ be a monoid, and let $U$ be a subset of $S\times S$.  For any $a,b\in S$ we have $(a,b)\in\l U\r$ if and only if there exists a sequence $$a=u_1s_1,\ v_1s_1=u_2s_2,\ \dots,\ v_ns_n=b$$ for some $n\in\N_0$, where $(u_i,v_i)\in U$ or $(v_i,u_i)\in U,$ and $s_i\in S,$ for each $i\in\{1,\dots,n\}.$
\end{Lem}

A sequence of the form given in Lemma \ref{lem:sequence} is referred to as a {\em $U$-sequence from $a$ to $b$ of length $n$}.  
In the case that $n=0,$ we interpret the $U$-sequence as $a=b.$ 

The universal relation $\w_S=S\times S$ is certainly a right congruence on $S.$  When viewing this relation as a right congruence, we shall denote it by $\w_S^r$.  Consider a generating set $U$ of $\w_S^r$.  For any $a,b\in S,$ let $d_U^r(a,b)$ denote the least $n\in\N_0$ such that there is a $U$-sequence from $a$ to $b$ of length $n.$  It is easy to see that $d_U^r : S\times S\to\N_0$ is a metric on $S.$

\begin{Def}
Let $S$ be a monoid such that $\w_S^r$ is finitely generated.  For each finite generating set $U$ of $\w_S^r$, we call the diameter of the metric space $(S,d_U^r)$ the {\em right $U$-diameter} of $S$ and denote it by $D_r(S;U)$; that is, 
\[D_r(S;U)=\sup\{d_U^r(a,b): a,b\in S\}.\]
We define the {\em right diameter} of $S$ to be
$$D_r(S)=\min\{D_r(S;U): \w_S^r=\l U\r\text{ where $U$ is finite}\}.$$
\end{Def}

Note that if $U$ and $U'$ are two finite generating sets for $\w_S^r,$ then $D_r(S;U)$ is finite if and only if $D_r(S;U')$ is finite \cite[Lemma 2.5]{Dandan}.  

For convenience, we shall often abuse terminology by saying that $\w_S^r$ is generated by a subset $V$ of $S$ to mean that $\w_S^r$ is generated by $V\times V$, and also write $D_r(S;V)$ in place of $D_r(S;V\times V).$  We note that $D_r(S)$ is the minimum of all $D_r(S;V)$ over finite $V\subseteq S$ with $\w_S^r=\l V\times V\r.$


%
%
%

As mentioned in Section \ref{sec:intro}, there is an important connection between the notion of diameter and that of diagonal acts.  Recall that the {\em diagonal right $S$-act} is the set $S\times S$ with right action given by $(a,b)c=(ac,bc).$  It is {\em generated} by a set $U\subseteq S\times S$ if $S\times S=US:=\{us : u\in U, s\in S\},$ and is {\em finitely generated} if it is generated by a finite set.  

\begin{Prop}\label{prop:dp}\cite[Proposition 3.6]{Gould:23}
A non-trivial monoid $S$ has right diameter 1 if and only if the diagonal right $S$-act is finitely generated.
\end{Prop}

\begin{Rem}\label{rem:finite}
A monoid has right (or left) diameter 0 if and only if it is trivial.  Every finite, non-trivial monoid has right (and left) diameter 1.
\end{Rem}

\section{Preliminaries on Chains and their Endomorphisms}\label{sec:chains}

\subsection{Chains}

A {\em chain} is a set equipped with a {\em total order}, i.e.\ a partial order under which any two elements are comparable.\footnote{In the literature, there are a number of other names used for chains, including totally ordered sets, tosets, linearly ordered sets and losets.}  

Let $\bf{C}=(C,\leq)$ be a chain.  Where no confusion arises, we identify $\bf{C}$ with $C.$  A {\em subchain} of $C$ is a subset $D\subseteq C$ totally ordered under the restriction of $\leq$ to $D.$  A subchain of $C$ is called {\em proper} if its underlying set is not equal to $C.$  An {\em interval} of $C$ is a subchain $D$ of $C$ such that if $a\leq x\leq b$ with $a,b\in D$ then $x\in D$.
In Definition \ref{def:int} all the subsets defined are intervals.

\begin{Def}\label{def:int} 
Let $C$ be a chain, and let $a,b\in C$ with $a\leq b$.  We write
\[\begin{array}{rclrcl}
(a,b)&=&\{x\in C : a<x<b\},\quad&
[a,b)&=&\{x\in C : a\leq x<b\},\\
(a,b]&=&\{x\in C : a<x\leq b\},\quad&
[a,b]&=&\{x\in C : a\leq x\leq b\},\\
a^{\uparrow}&=&\{x\in C : a\leq x\},\quad&
a^{\downarrow}&=&\{x\in C : x\leq a\}.
\end{array}\]
\end{Def}

By $\bf{C}^*$ we denote the chain obtained from $\bf{C}$ by reversing the order; that is, $\bf{C}^*=(C,\geq).$  We let $\bf{N}=(\N,\leq)$, $\bf{Z}=(\Z,\leq)$, $\bf{Q}=(\Q,\leq)$ and $\bf{R}=(\R,\leq),$ where in each case $\leq$ is the natural order.  For $n\in\N$, we denote  the interval $[1,n]$ of $\bf{N}$ by $\bf{n}.$

Given a chain $a_1<a_2<\cdots$ of elements of $C,$ we will occasionally abuse notation by writing the set $\{a_1,a_2,\dots\}$ as $\{a_1<a_2<\cdots\}.$

For subsets $A$ and $B$ of $C,$ we write $A<B$ if $a<b$ for each $a\in A$ and $b\in B$; if $A=\{a\}$ we simply write this as $a<B.$  Similarly, we write $A<b,$ $A\leq B,$ etc.  Note that we assume $A<\es$ and $A>\es$ are (vacuously) true, and interpret $A<\es<B$ as $A<B.$
We say that \textit{$B$ covers $A$} if $A<B$ and there is no $x\in C$ with $A<x<B$. 

An element $b\in C$ is a {\em lower bound} of a subset $A\subseteq C$ if $b\leq A$.  There is at most one lower bound of $A$ belonging to $A$; if such an element exists, it is called the {\em minimum} of $A$ and is denoted by $\min(A).$  {\em Upper bounds} and {\em maximums} of subsets of $C$ are defined dually, and the maximum of $A\subseteq C,$ if it exists, is denoted by $\max(A).$  An element that is either a minimum or a maximum of $C$ is called an {\em endpoint}.


Given a chain $D$ disjoint from $C,$ we let $C+D$ denote the chain $(C\cup D,\leq)$ with order extending those of $C$ and $D$ and with $C<D$.  Hence $D$ covers $C$ in $C+D.$  We note that we will often abuse notation by writing $C+D$ for non-disjoint $C$ and $D,$ by which we mean $C+D'$ where $D'$ is a copy of $D$ disjoint from $C.$  

We call a subchain $A$ of $C$ an \textit{initial segment} if $C=A+D$ for some chain $D.$  We dually define \textit{terminal segments}. 
A \textit{cut} of $C$ is a pair $(A,B)$ such that $C=A+B$. 
A cut $(A,B)$ is said to be {\em proper} if both $A$ and $B$ are proper.  
A proper cut $(A,B)$ of $C$ is a \textit{gap} if $A$ has no maximum and $B$ has no minimum.  We note that if $(A,B)$ is a gap of $C,$ then $A$ contains a copy of $\bf{N},$ and $B$ contains a copy of $\bf{N}^*$, and hence $C$ contains a copy of $\bf{N}+\bf{N}^*$.
If $C$ contains no gaps then it is called \textit{(Dedekind) complete}. 

%

Clearly there is a one-to-one correspondence between the cuts of $C$ and the initial (or terminal) segments of $C.$  The set $\mathcal{D}(C)$ of all cuts of $C$ comes equipped with a total order $\leq_{DM}$ given by 
$$(A,B)\leq_{DM}(A',B')\,\Leftrightarrow\,A\subseteq A'\;(\Leftrightarrow  B'\subseteq B).$$
The chain $(\mathcal{D}(C),\leq_{DM})$ is complete and is known as the \textit{Dedekind-MacNeille completion of $C$} (with minimum $(\emptyset,C)$ and maximum $(C,\emptyset)$) \cite[Chapter 7]{Davey}. 

\subsection{Homomorphisms and one-sided units}

Let $C$ and $D$ be chains.  A {\em homomorphism} from $C$ to $D$ is a map $\th : C\to D$ that is {\em order preserving}, i.e.\ $x\th\leq y\th$ whenever $x\leq y.$  An {\em embedding} is an injective homomorphism, and an {\em isomorphism} is a bijective homomorphism.  We call $D$ a {\em quotient} of $C$ if there is a surjective homomorphism from $C$ onto $D.$  We note that every quotient of $C$ gives rise to a partition of $C$ into intervals, and vice versa. 

Given a homomorphism $\th : C\to D$ and a subchain $E\subseteq D,$ we denote by $E\th^{-1}$ the subchain $\{x\in C : x\th\in E\}$ of $C.$  For $y\in D,$ we abbreviate $\{y\}\th^{-1}$ to $y\th^{-1}$.

An {\em endomorphism} of $C$ is a homomorphism from $C$ to $C.$  The set $\End(C)$ of all endomorphisms of $C$ is a monoid under composition, whose identity $1_C$ is the identity map on $C.$  Injective elements of $\End(C)$ are called {\em self-embeddings} of $C.$

A map $\a\in\End(C)$ is {\em regular} if there exists some $\b\in\End(C)$ such that $\a=\a\b\a.$
The following result, due to Mora and Kemprasit, describes all regular endomorphisms of chains.

\begin{Thm}\label{thm:regular}\cite[Theorem 2.4]{Mora}
Let $C$ be a chain and $\a\in\End(C).$  Then $\a$ is regular if and only if the following conditions hold. 
\begin{enumerate}
\item[\emph{(i)}] If $\im\a$ has a lower bound in $C,$ then $\min(\im\a)$ exists. 
\item[\emph{(ii)}] If $\im\a$ has an upper bound in $C,$ then $\max(\im\a)$ exists.
\item[\emph{(iii)}] If $x\in C{\setminus}\!\im\a$ is neither an upper bound nor a lower bound of $\im\a,$ then at least one of $\max\{t\in\im\a : t<x\}$ and $\min\{t\in\im\a : t>x\}$ exists.
\end{enumerate} 
\end{Thm}


Given $x\in C,$ the {\em constant map} on $x$ is the map $c_x : C\to C$ given by $yc_x=x$ for all $y\in C.$  It is clear that each $c_x$ is a regular, indeed an idempotent, endomorphism of $C.$ 

A map $\a\in\End(C)$ is a {\em right unit} if there exists $\b\in\End(C)$ such that $\a\b=1_C$; the element $\b$ is called a {\em right inverse} of $\a,$ and $\a$ is said to be {\em right invertible}.  {\em Left units} and {\em left inverses} are defined dually.  Clearly $\a$ is a right (resp.\ left) unit of $\End(C)$ if and only if it is a left (resp.\ right) inverse of some $\b\in\End(C).$
It is also clear that right/left units are regular.  The following lemma describes the left units of $\End(C).$

\begin{Lem}\label{lem:leftunits}
Left units of $\End(C)$ are precisely the surjective endomorphisms of $C.$
\end{Lem}

\begin{proof}
Clearly left units of $\End(C)$ are surjective.  For the converse, let $\b$ be a surjective endomorphism of $C.$  For each $x\in C$ choose $y_x\in x\b^{-1}$, and define $\a : C\to C$ by $x\a=y_x$.  It is clear that $\a\in\End(C),$ and for any $x\in C$ we have $x\a\b=y_x\b=x,$ so that $\a\b=1_C$.  Thus $\b$ is a left unit of $\End(C).$
\end{proof}

Right units of $\End(C)$ have a more complicated description.  They are certainly self-embeddings, but the converse need not hold.  Using the above description of regular endomorphisms of $C,$ we quickly determine which self-embeddings are right invertible. 

\begin{Cor}\label{cor:1-1RU} 
Let $C$ be a chain and $\a\in\End(C).$  Then the following statements are equivalent.
\begin{enumerate}[leftmargin=*]
\item[\emph{(1)}] $\a$ is a right unit of $\End(C).$
\item[\emph{(2)}] $\a$ is both injective and regular. 
\item[\emph{(3)}] $\a$ is injective and the following conditions hold.
\begin{enumerate}
\item[\emph{(a)}] If $C$ has no minimum, then $\im\a$ has no lower bound in $C.$
\item[\emph{(b)}] If $C$ has no maximum, then $\im\a$ has no upper bound in $C.$
\item[\emph{(c)}] If $x\in C{\setminus}\!\im\a$ is neither an upper bound nor a lower bound of $\im\a,$ then at least one of $\max\{t\in\im\a : t<x\}$ and $\min\{t\in\im\a : t>x\}$ exists.
\end{enumerate}
\end{enumerate} 
\end{Cor}
 
\begin{proof}
(1)$\Leftrightarrow$(2).  We have already observed that right units of $\End(C)$ are injective and regular.  For the converse, suppose $\a\in\End(C)$ is injective and regular.  Then $\a=\a\b\a$ for some $\b\in S.$  For any $x\in C,$ we have $x\a=(x\a\b)\a,$ and hence $x=x\a\b$ by injectivity of $\a.$  Thus $\a\b=1_C$, as required.  

(2)$\Leftrightarrow$(3).  Let $\a$ be injective.  We need to show that $\a$ is regular if and only if the conditions (a)-(c) hold.  By Theorem \ref{thm:regular}, regularity of $\a$ is equivalent to the conditions (i)-(iii) of that theorem holding.  Since $\a\in S$ is injective, it is clear that $\min(\im\a)$ exists if and only if $C$ has a minimum, and hence (i) is equivalent to (a).  Similarly, (ii) is equivalent to (b).  Finally, (iii) is identical to (c).
\end{proof}

\begin{Rem}\label{rem:notRU} 
If $\a$ is a self-embedding of a chain $C$ that fails condition (3)(c) of Corollary \ref{cor:1-1RU} with witness $x\in C{\setminus}\!\im\a,$ then ($x^{\downarrow}\a^{-1},x^{\uparrow}\a^{-1}$) is a gap of $C.$
\end{Rem}

The following technical lemma will be crucial in our arguments to determine the exact value of the right diameter of $\End(C).$  

\begin{Lem}\label{lem:rightunitset}
Let $C$ be a chain, let $x$ be a fixed element of $C,$ let $\a$ be a right unit of $\End(C),$ and let $Q=\{\b\in\End(C) : \a\b=1_C\}$ be the set of right inverses of $\a.$  Then, for any $\b\in Q$ we have $$x\b\a=
\begin{cases}
\max(\im\a\cap x^{\downarrow})&\text{if }x\b\a\leq x\\
\min(\im\a\cap x^{\uparrow})&\text{if }x\b\a\geq x.
\end{cases}$$
Consequently, the set $\{x\b : \b\in Q\}$ has at most two elements.  
\end{Lem}

\begin{proof}
Let $Y_1=\im\a\cap x^{\downarrow}$ and $Y_2=\im\a\cap x^{\uparrow}$.  Then $\im\a=Y_1\cup Y_2$ with $Y_1\leq Y_2$.  Consider $\b\in Q.$  For any $y\in Y_1$, we have $y\leq x$ and $y=u\a$ for some $u\in C,$ so that
$$y=u\a=u\a\b\a=y\b\a\leq x\b\a.$$
Thus $Y_1\leq x\b\a.$  Hence, if $x\b\a\leq x$ (so that $x\b\a\in Y_1$) then $x\b\a=\max(Y_1).$  Similarly, if $x\b\a\geq x$ then $x\b\a=\min(Y_2).$

Now, since the sets $Y_1$ and $Y_2$ do not depend on any $\b\in Q,$ and $\a$ is injective, it follows that $|\{x\b : \b\in Q\}|\leq2.$ 
\end{proof}

\section{Endpoint-shiftability}\label{sec:es}

In this section we study the condition appearing in Theorem \ref{thm:B}.  We begin by providing equivalent formulations of this condition.

\begin{Prop}\label{prop:es}
For a chain $C$ with an endpoint $z,$ the following are equivalent:
\begin{enumerate}
\item[\emph{(1)}] $C$ is a quotient of $C{\setminus}\{z\}$;
\item[\emph{(2)}] there exists a left unit (surjective map) $\b\in\End(C)$ such that $z\b^{-1}\neq\{z\}$;
\item[\emph{(3)}] there exists a right unit $\a\in\End(C)$ such that $z\a\neq z$.
\end{enumerate}
\end{Prop}

\begin{proof}
(1)$\Rightarrow$(2).  By assumption, there exists a surjective homomorphism $\th : C{\setminus}\{z\}\to C.$  Extend $\th$ to a map $\b : C\to C$ by setting $z\b=z.$  Then $\b$ is a surjective endomorphism of $C,$ and hence a left unit of $\End(C)$ by Lemma \ref{lem:leftunits}, and $z\b^{-1}\neq\{z\}.$

(2)$\Rightarrow$(3).  
As in the proof of Lemma \ref{lem:leftunits}, we define a right unit $\a\in\End(C)$ by $x\a=y_x\in x\b^{-1}$ ($x\in C$), but choose $y_z\neq z.$

(3)$\Rightarrow$(1).  Let $\b$ be any right inverse of $\a.$  Then $\b$ is a surjective endomorphism of $C.$  Since $(z\a)\b=z$ and $z\a\neq z,$ it follows that $\b$ restricts to a surjective homomorphism $C{\setminus}\{z\}\to C,$ and hence $C$ is a quotient of $C{\setminus}\{z\}.$
\end{proof}

\begin{Def}
We say that a chain $C$ is {\em endpoint-shiftable} if it has an endpoint $z\in C$ and satisfies any (and hence all) of the three equivalent conditions of Proposition \ref{prop:es}.  If, in fact, $C$ is endpoint-shiftable with respect to a minimum (resp.\ maximum), then we call it {\em min-shiftable} (resp.\ {\em max-shiftable}).
\end{Def}


%

The next few results concern min-shiftability; of course, there are obvious duals for max-shiftability.  First, we provide several equivalent characterisations for a chain to be min-shiftable.


\begin{Prop}\label{prop:ms}
For a chain $C$ with a minimum, the following are equivalent:
\begin{enumerate}
\item[\emph{(1)}] $C$ is min-shiftable;
\item[\emph{(2)}] $C^*$ is max-shiftable;
\item[\emph{(3)}] $C$ is a quotient of one of its proper terminal segments;
\item[\emph{(4)}] for every chain $D,$ the chain $C+D$ is min-shiftable;
\item[\emph{(5)}] for each $n\in\N,$ the chain $\bf{n}+C$ is a quotient of $C$;
\item[\emph{(6)}] for some $n\in\N,$ the chain $\bf{n}+C$ is a quotient of $C$;
\item[\emph{(7)}] for each $n\in\N,$ the chain $\bf{n}+C$ is min-shiftable;
\item[\emph{(8)}] for some $n\in\N,$ the chain $\bf{n}+C$ is min-shiftable.
\end{enumerate}
\end{Prop}

\begin{proof}
Clearly $(1)\Leftrightarrow(2),$ $(5)\Rightarrow(6)$ and $(7)\Rightarrow(8).$  Let $z=\min(C).$

$(1)\Leftrightarrow(3).$  The forward implication follows immediately from the definition of being min-shiftable together with Proposition \ref{prop:es}(1).  For the converse, suppose that $C$ is a quotient of a proper terminal segment $T$ of $C.$  Let $\phi : T\to C$ be a surjective homomorphism, and extend $\phi$ to a surjective homomorphism $\th : C\to C$ by letting $(C{\setminus}T)\th=\{z\}.$  Then $z\th^{-1}\neq\{z\}$ as $T\neq C.$  Thus $C$ is min-shiftable.


$(1)\Leftrightarrow(4).$  Suppose that $C$ is min-shiftable.  Then there exists a surjective homomorphism $\phi : C{\setminus}\{z\}\to C.$  For any chain $D,$ clearly $z=\min(C+D),$ and $\phi$ can be extended to a surjective homomorphism $\th : (C+D){\setminus}\{z\}\to C+D$ by letting $x\th=x$ for all $x\in D,$ so that $C+D$ is min-shiftable.  For the converse, put $D=\emptyset.$

$(1)\Rightarrow(5).$  Since $C$ is min-shiftable, there exists a surjective map $\b\in\End(C)$ such that $z\b^{-1}\neq\{z\}.$  It is straightforward to show that $\g:=\b^n : C\to C$ is a surjective homomorphism with $|z\g^{-1}|\geq n+1.$  Choose a surjective homomorphism $\phi : z\g^{-1}\to\bf{n}+\{z\},$ and define a map
$$\th : C\to\bf{n}+C,\ x\mapsto
\begin{cases}
x\phi&\text{if }x\in z\g^{-1}\\
x\g&\text{ otherwise.}
\end{cases}$$
Clearly $\th$ is a surjective homomorphism, as required.


$(5)\Rightarrow(7)$ and $(6)\Rightarrow(8)$ follow from the already-established implication $(1)\Rightarrow(3)$.

$(8)\Rightarrow(3).$  There exists a surjective homomorphism $\b : \bf{n}{\setminus}\{1\}+C\to\bf{n}+C.$  It follows from the surjectivity of $\b$ that $i\b<i$ for each $i\in\{2,\dots,n\},$ and hence $n\b^{-1}\subseteq C,$ which implies that $z\b\in\bf{n}.$  It now follows that $C\b^{-1}$ is a proper terminal segment of $C,$ and that $C$ is a quotient of $C\b^{-1}$.
\end{proof}

The next result provides another means of constructing min-shiftable chains.  First, recall that the Cartesian product $C\times D$ of two chains $C$ and $D$ is a chain under the {\em lexicographical order}, given by $(x,y)\leq(x',y')$ if and only $x<x'$ or ($x=x'$ and $y\leq y'$).  Note that $(z,w)=\min(C\times D)$ if and only if $z=\min(C)$ and $w=\min(D).$

\begin{Prop}\label{prop:product}
Let $C$ and $D$ be chains with minima.  If $C$ is isomorphic to one of its proper terminal segments, or if $D$ is min-shiftable, then $C\times D$ is min-shiftable.
\end{Prop}

\begin{proof}
Let $z=\min(C).$  Since $C\times D\cong D+(C{\setminus}\{z\}\times D),$ if $D$ is min-shiftable then, by Proposition \ref{prop:ms}, so is $C\times D.$  Suppose now that $C$ is isomorphic to a proper terminal segment $T$ of $C,$ and let $\phi : C\to T$ be an isomorphism.  Then $C\times D$ is isomorphic to its proper terminal segment $T\times D$ via $(x,y)\mapsto(x\phi,y).$  Thus, by Proposition \ref{prop:ms}, the chain $C\times D$ is min-shiftable.
\end{proof}

We now present a number of natural examples of min-shiftable chains.  Together with Propositions \ref{prop:ms} and \ref{prop:product}, these examples generate a large class of endpoint-shiftable chains.  First, recall that a chain $C$ is {\em scattered} if $\bf{Q}$ does {\em not} embed into $C$; otherwise, $C$ is {\em non-scattered}.

\begin{Prop}\label{prop:msexamples}~
\begin{enumerate}[leftmargin=*]
\item[\emph{(1)}] Every infinite well-ordered chain is min-shiftable.
\item[\emph{(2)}] If $C$ is a countable, non-scattered chain with a minimum, then $C$ is min-shiftable.
\item[\emph{(3)}] If $D$ is a subchain of $\bf{1}+\bf{Z}+\bf{1}$ with both a minimum and a maximum, then $D+\bf{R}$ is min-shiftable.
\end{enumerate}
In particular, the following chains are min-shiftable: $$\bf{N},\quad\bf{1}+\bf{Q},\quad\bf{1}+\bf{N}^*+\bf{Q},\quad\bf{1}+\bf{Z}+\bf{Q},\quad\bf{1}+\bf{R},\quad\bf{1}+\bf{N}^*+\bf{R},\quad\bf{1}+\bf{Z}+\bf{1}+\bf{R}.$$
\end{Prop}

\begin{proof}
(1) An infinite well-ordered chain is isomorphic to $\bf{N}+D$ for some well-ordered chain $D$; this is folklore, but follows from \cite[Theorem 2.8]{Jech}.  
Moreover, $\bf{N}+D$ is isomorphic to $\bf{N}{\setminus}\{1\}+D$ via the map given by $n\mapsto n+1$ ($n\in\N$) and $x\mapsto x$ ($x\in D$).

(2) By \cite[Proposition 17]{Camerlo}, we may write $C$ as $C=C_0+C'+C_1$ where $C_0$ and $C_1$ are scattered and $\bf{Q}$ is a quotient of $C'$.  Note that $C_0$ is non-empty, with $\min(C)=\min(C_0),$ while $C_1$ is possibly empty.  By \cite[Proposition 16(1)]{Camerlo}, every countable chain is a quotient of $\bf{Q}$; in particular, $C_0+C'$ is a quotient of $\bf{Q}.$  It follows that $C_0+C'$ is a quotient of its proper terminal segment $C'$, and is hence min-shiftable by Proposition \ref{prop:ms}.  Thus, by Proposition \ref{prop:ms} again, $C=(C_0+C')+C_1$ is min-shiftable.

(3) By definition, $D=\bf{1}+D'$ for some terminal segment $D'$.  It is easy to see that $D$ is a quotient of $\bf{R}.$  It follows that $D$ is also a quotient of $D'+\bf{R}+\bf{1},$ since we can map $D'$, if non-empty, to $\min(D),$ and $\bf{1}$ to $\max(D).$  It can also easily be shown that $\bf{R}\cong\bf{R}+\bf{1}+\bf{R}.$  Consequently, $D+\bf{R}$ is a quotient of $D'+\bf{R}\cong(D'+\bf{R}+\bf{1})+\bf{R}.$  
\end{proof}

One can easily find natural examples of chains that are {\em not} endpoint-shiftable.  Certainly chains without endpoints are not endpoint-shiftable.
Further examples include chains with endpoints that do not embed into any of their proper initial or terminal segments.  Such chains certainly include all finite chains, and also the chains $\bf{1}+\bf{Z},$ $\bf{Z}+\bf{1}$ and $\bf{1}+\bf{Z}+\bf{1}.$


Our next example is a chain with a minimum (but no maximum) that is not min-shiftable but has self-embeddings that move the minimum.

\begin{Ex}\label{ex:1+Z+R}
Let $C=\bf{1}+\bf{Z}+\bf{R},$ and let $z=\min(C).$  Certainly $C$ has self-embeddings that move $z,$ since $C$ may be embedded into $\bf{R}.$  Suppose for a contradiction that $C$ is min-shiftable, and let $\a\in\End(C)$ be a right unit with $z\a>z.$  
If there were some $x\in\bf{Z}$ with $x\a\in\bf{Z},$ then $[z,x]\a=[z\a,x\a]$ would be a finite interval of $\bf{Z},$ but $[z,x]$ is infinite, so this would contradict that $\a$ is injective.  Thus $\bf{Z}\a\cap\bf{Z}=\emptyset,$ and hence $\bf{Z}\a\subseteq\bf{R}$. 
Let 
$$A=\{x\in\bf{R} : x\leq y\text{ for some }y\in\bf{Z}\a\}\quad\text{ and }\quad B=\{x\in\bf{R} : x>\bf{Z}\a\}.$$
Clearly $\bf{R}=A+B,$ with $\bf{Z}\a\subseteq A$ and $\bf{R}\a\subseteq B,$ and $A$ has no maximum as $\bf{Z}$ has no maximum.  Since $\bf{R}$ is complete, the proper cut $(A,B)$ is not a gap, and hence $B$ has a minimum, say $x.$  We cannot have $x\in\bf{R}\a,$ since $\bf{R}$ has no minimum, so $x\notin\im\a.$  But then $\{t\in\im\a : t<x\}=(\bf{1}+\bf{Z})\a$ has no maximum, and $\{t\in\im\a : t>x\}=\bf{R}\a$ has no minimum, contradicting Corollary \ref{cor:1-1RU}.
\end{Ex}

\vspace{0.2em}
\begin{Rem}~
\begin{enumerate}[leftmargin=*]
\item By Proposition \ref{prop:msexamples}, the chain $\bf{1}+\bf{Z}+\bf{Q}$ is min-shiftable.  The key reason that the trick used in Example \ref{ex:1+Z+R} does not work for $\bf{1}+\bf{Z}+\bf{Q}$ is that $\bf{Q}$ is not complete.  
\item Clearly $\bf{1}+\bf{Z}+\bf{R}$ is a non-scattered chain with a minimum.  Thus, Example \ref{ex:1+Z+R} demonstrates the necessity of the chain $C$ of Proposition \ref{prop:msexamples}(2) being countable.
\item By a similar argument as that of Example \ref{ex:1+Z+R}, the chain $\bf{1}+\bf{Z}+\bf{R}+\bf{Z}+\bf{1}$ is not endpoint-shiftable but has self-embeddings that move both endpoints.
\end{enumerate}
\end{Rem}

\section{Proofs of the Main Results}\label{sec:main}


In this section we prove the main results of the paper.
For what follows, recall that the constants maps $c_x$ of the chain $C$ belong to $\End(C),$ and note that for each $\th\in\End(C)$ we have $\th c_x=c_x$ and $c_x\th=c_{x\th}$.  We now quickly prove Theorem \ref{thm:A}.


\begin{proof}[Proof of Theorem \ref{thm:A}]
Let $S=\End(C).$  By \cite[Theorem 4.5.4]{Gallagher}, the diagonal left $S$-act is not finitely generated.  Thus, by the dual of Proposition \ref{prop:dp}, the left diameter of $S$ is not 1.  
Now pick any $x\in C,$ and let $U=\{1_C,c_x\}.$  For any $\th,\phi\in S,$ we have a $U$-sequence
$$\th=\th1_C,\ \th c_x=\phi c_x,\ \phi1_C=\phi$$
of length 2, so that $D_l(S;U)\leq2.$  Thus $D_l(S)=2.$
\end{proof}

The proof of Theorem \ref{thm:B} is considerably more complicated.  We begin by showing that the right diameter of $\End(C)$ is either 2 or 3.

\begin{Prop}\label{prop:chain,right}
If $C$ is an infinite chain, then $D_r(\End(C))$ is either 2 or 3.
\end{Prop}

\begin{proof}
Let $S=\End(C).$  By \cite[Theorem 4.5.5]{Gallagher} and Proposition \ref{prop:dp}, the monoid $S$ does not have right diameter 1.  In fact, this can also be deduced quickly from Lemma \ref{lem:rightunitset}, as follows.  We give the proof since it is short and introduces methods that will be used later.

Suppose for a contradiction that $D_r(S;U)=1$ for some finite set $U\subseteq S.$  Fix any $x\in C.$  For each $\a,\b\in U$ where $\a$ is a right unit, let $J_{\a,\b}=\{(x\b)\g : \g\in S, \a\g=1_C\},$ and let $J$ be the union of all the $J_{\a,\b}$.  The set $J$ is finite since $U$ is finite and, by Lemma \ref{lem:rightunitset}, each $J_{\a,\b}$ is finite.  Now pick any $y\in C{\setminus}J.$  Then, by assumption, there exist $\a,\b\in U$ and $\g\in S$ such that $1_C=\a\g$ and $c_y=\b\g.$  But then $y=(x\b)\g\in J,$ a contradiction.  Thus $S$ has does not have right diameter 1.

We now prove that $D_r(S)\leq3.$  Fix any $x,y\in C$ with $x<y,$ and let $U=\{1_C,c_x,c_y\}.$  We claim that $\w_S^r=\l U\r$ with $D_r(S;U)\leq3.$  Consider $\th,\phi\in S.$  We may assume without loss of generality that $x\th\leq x\phi.$  
Define a map 
$$\g : C\to C,\ w\mapsto
\begin{cases}
x\th&\text{if }w\leq x\\
x\phi&\text{if }w>x.
\end{cases}$$
It is easy to see that $\g\in S.$  Moreover, we have a $U$-sequence
$$\th=1_C\th,\ c_x\th=c_x\g,\ c_y\g=c_x\phi,\ 1_C\phi=\phi$$
of length 3, as required.
\end{proof}

Next, we prove that the endomorphism monoid of an endpoint-shiftable chain has right diameter 2.

\begin{Prop}\label{prop:endpoint-shiftable}
If $C$ is an endpoint-shiftable chain, then $D_r(\End(C))=2.$  
\end{Prop}

\begin{proof}
By symmetry we may assume that $C$ is min-shiftable. 
Let $S=\End(C).$  By Proposition \ref{prop:chain,right} we have $D_r(S)\geq2.$  Now let $z$ be the minimum of $C,$ and let $\a\in S$ be a right unit such that $z\a>z.$  Letting $U=\{\a,c_z\},$ we claim that $\w_S^r=\l U\r$ and $D_r(S;U)=2.$  So, consider $\th\in S.$  Since $\a$ is a right unit of $S,$ there exists $\b\in S$ such that $\a\b=1_C$.  Define a map 
$$\g : C\to C, x\mapsto 
\begin{cases}
x\b\th&\text{ if }x\geq z\a,\\
z&\text{ otherwise.}
\end{cases}$$
It is straightforward to show that $\g\in S.$  For each $x\in C$ we have $x\a\g=x\a\b\th=x\th,$ so 
$$\th=\a\g,\ c_z\g=c_z$$
is a $U$-sequence of length 1 from $\th$ to $c_z$.  Since $\th$ is arbitrary, it follows that $D_r(S;U)\leq2.$  Thus $D_r(S)=2.$
\end{proof}

\begin{Rem}
It is perhaps interesting that in each of the proofs of Theorem \ref{thm:A} and Propositions \ref{prop:chain,right} and \ref{prop:endpoint-shiftable}, to obtain the upper bound on the diameter, only one `skeleton' is required to connect any pair $\th,\phi\in S=\End(C),$ where the {\em skeleton} of a $U$-sequence ($U\subseteq S$) refers to the sequence of elements of $U$ appearing in the $U$-sequence.  For instance, the skeleton in the proof of Proposition \ref{prop:chain,right} is $(1_C,c_x,c_x,c_y,c_x,1_C)$.
\end{Rem}


 
We now turn to prove the converse of Proposition \ref{prop:endpoint-shiftable}.  Our first step is to show that in order for the endomorphism monoid of a chain to have right diameter 2, it is necessary that the chain have an endpoint.

\begin{Prop}\label{prop:noends}
If $C$ is a chain with no endpoints, then $D_r(\End(C))=3.$
\end{Prop}

\begin{proof} 
Let $S=\End(C),$ and let $R$ denote the set of right units of $S.$  Aiming for a contradiction, suppose that $D_r(S;U)=2$ for some finite set $U\subseteq S.$

For each $\d,\e\in U$ (not necessarily distinct) define
\[B(\d,\e)=\{x\in C : x\d\leq x\e\}.\]
Since $C=B(\d,\e)\cup B(\e,\d),$ at least one of $B(\d,\e)$ or $B(\e,\d)$ is unbounded above.  Let
$$V=\{(\d,\e)\in U\times U : B(\d,\e)\text{ is unbounded above}\}.$$
For each $(\d,\e)\in V$ fix some $x(\d,\e)\in B(\d,\e),$  and let $m=\max\{x(\d,\e) : (\d,\e)\in V\}.$  Then $C=m^{\downarrow}+D,$ where $D=\{c\in C : c>m\}.$  By the unboundedness assumption, for each $(\d,\e)\in V$ we may fix some $y(\d,\e)\in B(\d,\e)\cap D$.  Let
$$X=\{x(\d,\e)\b : (\d,\e)\in V,\,\b\in U\} \AND Y=\{y(\d,\e)\b : (\d,\e)\in V,\,\b\in U\},$$
and note that $X$ and $Y$ are finite.
Since $C$ has no endpoints, by Corollary \ref{cor:1-1RU} the image of any map in $R$ is unbounded both above and below in $C.$  Thus, for each $\a\in U\cap R$ we may fix some $p_\a,q_\a\in C$ with $p_\a<q_\a$ such that $p_\a\a<X$ and $q_\a\a>Y.$  We then let
$$p=\min\{p_\a : \a\in U\cap R\} \AND q=\max\{q_\a : \a\in U\cap R\},$$
so that $p\a<X$ and $q\a>Y$ for each $\a\in U\cap R.$
Now fix some $u,v\in C$ with $u<p<q<v,$ and define $\th\in S$ by $m^{\downarrow}\th=\{u\}$ and $D\th=\{v\}.$  Since $D_r(S)=2,$ there exists a $U$-sequence
$$1_C=\a\g_1,\ \b\g_1=\d\g_2,\ \e\g_2=\th.$$
Note that $\a$ is a right unit.  We consider two separate cases.

\textit{Case 1}.
Suppose first that $(\d,\e)\in V$.  Let $x=x(\d,\e)$ and $y=y(\d,\e).$ Then, since $x\d\leq x\e$ (as $x\in B(\d,\e)$) and $p\a<x\b\in X,$ we have
\[u=x\th=x\e\g_2\geq x\d\g_2=x\b\g_1\geq p\a\g_1=p1_C=p,\]
and this contradicts $u<p.$

\textit{Case 2}.  Finally, suppose that $(\e,\d)\in V.$  Let $x=x(\e,\d)$ and $y=y(\e,\d).$  Then, since $y\e\leq y\d$ (as $y\in B(\e,\d)$) and $q\a>y\b\in Y,$ we have
\[v=y\th=y\e\g_2\leq y\d\g_2=y\b\g_1\leq q\a\g_1=q1_C=q,\]
and this contradicts $v>q.$
\end{proof}

To complete the proof of Theorem \ref{thm:B}, we show in Proposition \ref{prop:notES} that $D_r(\End(C))=3$ if $C$ has an endpoint but is not endpoint-shiftable.  First, we establish some technical lemmas.  In what follows, a map $\a\in\End(C)$ is said to {\em fix} a subset $A\subseteq C$ {\em pointwise} if $x\a=x$ for all $x\in A.$

\begin{Lem}\label{lem:CD}
Let $C$ be an infinite chain that has an endpoint but is not endpoint-shiftable.  Then $C=M+D+N$ for some finite chains $M$ and $N$ (where at least one of $M$ and $N$ is non-empty) and some chain $D$ with no endpoints.  Moreover, every right/left unit of $\End(C)$ fixes $M\cup N$ pointwise and restricts to a right/left unit of $\End(D).$ 
\end{Lem}

\begin{proof}  
Let $S=\End(C)$.  If $C$ has no minimum, then we define $M=\es.$  So now suppose that $C$ has a minimum, say $z_1$.  Then $C=\{z_1\}+C_1$, where $C_1=C{\setminus}\{z_1\}$.  If $C_1$ has a minimum, say $z_2$, then $C=\{z_1<z_2\}+C_2$, where $C_2=C_1{\setminus}\{z_2\}$.  This process must terminate at some point, for otherwise we would have
\[C=\{z_1<z_2<z_3<\cdots\}+C'\cong\bf{N}+C',\]
which is min-shiftable by Propositions \ref{prop:ms} and \ref{prop:msexamples}(1).  So, there exists $m\in\N$ such that $C=\{z_1<\dots<z_m\}+C_m$ where $C_m$ has no minimum.  We let $M=\{z_1<\dots<z_m\}.$  The definition of $N$ is analogous, and then $C=M+D+N$ where $D=C_m{\setminus}N$ has no endpoints.

Now consider $\a,\b\in S$ such that $\a\b=1_C$, and let $\a',\b'$ be the restrictions of $\a,\b$ to $D,$ respectively.  We show that $\a$ and $\b$ both fix $M\cup N$ pointwise, and $D\a,D\b\subseteq D.$  It then follows that $\a'\b'=1_D$, as desired.

First, we show that $\b$ fixes $M\cup N$ pointwise and $D\b\subseteq D.$  To do so, suppose to the contrary that there exists $u\in C$ such that $u\b\in M\cup N$ with $u\neq u\b.$  We may assume without loss of generality that $u\b\in M.$  With $M=\{z_1<\cdots<z_m\}$ as above, suppose that $u\b=z_k$.  Since $\b$ is surjective and order preserving, it follows that $u>z_k$, and hence $z_k^{\uparrow}$ is a quotient of its proper terminal segment $u^{\uparrow}$.  
Thus, by Proposition \ref{prop:ms}, the chain $z_k^{\uparrow}$ is min-shiftable.  
Now, either $C=z_k^{\uparrow}$ (and is hence min-shiftable) or else
$$C=\{z_1<\dots<z_{k-1}\}+z_k^{\uparrow}\cong(\bf{k-1})+z_k^{\uparrow},$$
which is min-shiftable by Proposition \ref{prop:ms}.  We have reached a contradiction. 

Now, if $x\in M\cup N$ then $x=(x\a)\b=x\a,$ as $\b$ fixes $M\cup N$ pointwise and $D\b\subseteq D.$  Thus $\a$ fixes $M\cup N$ pointwise.  Since $\a$ is injective, it follows that $D\a\subseteq D,$ as required.
\end{proof}

\begin{Lem}\label{lem:shifting} 
Let $C,$ $D$ and $E$ be pairwise-disjoint chains with $D$ infinite, and let $K$ and $L$ be finite subsets of $C$ and $E,$ respectively.  Then there exists a homomorphism $\th : C+D+E\to D$ such that:
\begin{enumerate}
\item[\emph{(1)}] $\th$ is injective on $K\cup L$;
\item[\emph{(2)}] if $C=\es$ and $D$ has no minimum, then $\im\th$ is unbounded below;
\item[\emph{(3)}] if $E=\es$ and $D$ has no maximum, then $\im\th$ is unbounded above.
\end{enumerate}
\end{Lem}


\begin{proof} 
If $C=E=\es,$ let $\th=1_D$.  Suppose now that $C\cup E\neq\es.$  Pick $x\in D$ such that, letting $A=D\cap x^{\downarrow}$ and $B=(D\cap x^{\uparrow}){\setminus}\{x\},$ we have $|A|>|K|$ and $|B|>|L|.$  Clearly $D=A+B.$
If both $C$ and $E$ are non-empty, let $C\th\subseteq A,$ $D\th=\{x\}$ and $E\th\subseteq B$, with $\th$ preserving order and being injective on $K\cup L.$  If $C=\es$ and $D$ has no minimum, let $\th$ fix $A$ pointwise, let $B\th=\{x\},$ and let $E\th\subseteq B$ with $\th$ preserving order and being injective on $L.$  The case that $E=\es$ and $D$ has no maximum is symmetrical.
\end{proof}

\begin{Lem}\label{lem:trim}
Let $C=M+D+N$ be a chain where $M$ and $N$ are finite and $D$ has no endpoints.  Let 
$$(A_1,B_1)<_{DM}(A_2,B_2)<_{DM}\cdots<_{DM}(A_n,B_n)$$ be a finite chain of gaps of $C,$ and let $F$ be a finite subset of $D.$  Then there exist terminal segments $T_k$ $(\neq\es)$ of $A_k$ $(1\leq k\leq n)$, initial segments $I_k$ $(\neq\es)$ of $B_k$ $(1\leq k\leq n)$, and intervals $D_p$ of $D$ $(1\leq p\leq n+1)$ such that 
$$D=D_1+T_1+I_1+D_2+T_2+I_2+\cdots+T_n+I_n+D_{n+1},\quad\: F\subseteq D_1+D_2+\cdots+D_{n+1},$$ 
and for each $p\in\{1,\dots,n+1\}$ the set $F\cap D_p$ contains no endpoints of $D_p$.
\end{Lem}

\begin{proof}
First note that
$$D=(A_1\cap D)+(B_1\cap A_2)+(B_2\cap A_3)+\cdots+(B_{n-1}\cap A_n) +(B_n\cap D),$$
with each summand being unbounded above and below.  For each $k\in\{2,\dots,n\}$ choose $x_k,y_k\in B_{k-1}\cap A_k$ such that $x_k<B_{k-1}\cap A_k\cap F<y_k$.  Also let $y_1\in A_1\cap D$ and $x_{n+1}\in B_n\cap D$ 
be such that $A_1\cap F<y_1$ and $x_{n+1}<B_n\cap F.$  We then define
$$D_1=y_1^{\downarrow}\cap D,\quad D_k=[x_k,y_k]\,\ (2\leq k\leq n)\quad\text{and}\quad D_{n+1}=x^{\uparrow}_{n+1}\cap D.$$
These then determine appropriate subchains $T_k,I_k$ ($1\leq k\leq n$) such that
$$A_1\cap D=D_1+T_1,\:B_{k-1}\cap A_k=I_{k-1}+D_k+T_k\ (2\leq k\leq n)\:\text{ and }\:B_n\cap D=I_n+D_{n+1}.$$
Verification of the claimed properties is routine.
\end{proof}
 
\begin{Prop}\label{prop:notES} 
If $C$ is an infinite chain that has an endpoint but is not endpoint-shiftable, then $D_r(\End(C))=3.$
\end{Prop} 
 
\begin{proof} 
By Lemma \ref{lem:CD} we may write $C=M+D+N$, where $M$ and $N$ are finite and $D$ is unbounded below and above.  We assume without loss of generality that $M\neq\es$ with minimum element $z.$
Let $S=\End(C),$ and let $R$ be the set of right units of $S.$  Suppose for a contradiction that $D_r(S;U)=2$ for some finite $U\subseteq S.$ 

Now, it follows from Lemma \ref{lem:rightunitset} that the set
$$J=\{(z\a)\g : \a\in U,\,\g\text{ is a right inverse of some }\b\in U\cap R\}$$
has size at most $2|U|^2.$
For each non-injective $\a\in U$ fix some $a_\a,b_\a\in C$ with $a_\a\neq b_\a$ and $a_\a\a=b_\a\a,$ and let 
$$K=\{a_\a,b_\a : \a\in U,\,\a\text{ is not injective}\}.$$  
Let $F=M\cup N\cup J\cup K,$ and note that $F$ is finite.  

Now define a subset $T$ of $U$ as follows.  If $N\neq\es,$ let $T$ be the set of all injective $\d\in U{\setminus}R.$  If $N=\es$, let $T$ be the set of all injective $\d\in U{\setminus}R$ such that $\im\d$ is unbounded above. 

The remainder of the proof splits into two cases.

\textit{Case 1}: $T=\es$.  Choose $x\in D$ such that $x\geq F\cap D.$  (If $F\cap D\neq\emptyset,$ we can take $x=\text{max}(F\cap D)$; otherwise, we can take $x$ to be any element of $D.$)  Let $C'=D{\setminus}x^{\downarrow}$, so that $C=x^{\downarrow}+C'+N.$  Then $F\cap C'=\es.$  By Lemma \ref{lem:shifting}, there exists $\th\in S$ with $\im\th\subseteq C'$ such that $\th$ is injective on $F,$ and if $N=\es$ then $\im\th$ is unbounded above in $C'.$  Since $D_r(S;U)=2,$ there exists a $U$-sequence
$$\th=\a\g,\ \b\g=\a'\g',\ \b'\g'=1_C.$$ 
The map $\a\in U$ is injective, since otherwise
$a_\a\th=a_\a\a\g=b_\a\a\g=b_\a\th,$ contradicting that $\th$ is injective on $F\supseteq K.$
 Moreover, if $N=\es$ then $\im\a$ is unbounded above, since if $\im\a\leq q$ then $\im\th=\im\a\g\leq q\g,$ contradicting that $\im\th$ is unbounded above.  Since $T=\emptyset,$ we conclude that $\a\in R.$  Then, since $C$ is not endpoint-shiftable, it follows from Lemma \ref{lem:CD} that $\a$ fixes $M\cup N$ pointwise.  Thus $x\g=x\th$ for all $x\in M\cup N.$  Now, we may pick $y\in C$ such that $z\b\leq y\a.$  Indeed, if $N=\es$ then, as noted above, the set $\im\a$ is unbounded above; otherwise, the chain $C$ has a maximum, say $w,$ and then $z\b\leq w=w\a$ (as $w\in N$).  Therefore, since $z\in M,$ we have $z\a=z\leq z\b\leq y\a,$ and hence
$$z\th=z\a\g\leq z\b\g\leq y\a\g=y\th.$$
Since $\im\th\subseteq C'$, and $C'$ is an interval, it follows that $z\b\g\in C'.$  But $z\b\g=z\a'\g'\in J\subseteq F,$ contradicting the fact that $C'\cap F=\es.$

\vspace{0.5em}
\textit{Case 2}: $T\neq\es.$  Since each $\d\in T$ is injective but not a right unit, it must fail at least one of the conditions (a)-(c) of Corollary \ref{cor:1-1RU}(3).  But $C$ has a minimum, so (a) vacuously holds.  If $C$ has no maximum then $N=\es,$ and hence, by the definition of $T,$ if $\d\in T$ then $\im\d$ is unbounded above, so $\d$ satisfies (b).  Thus, each $\d\in T$ fails condition (c), and hence, by Remark \ref{rem:notRU}, there exists $x_\d\in C{\setminus}\!\im\d$ such that $x_\d>z\d$ and 
$$\mathcal{G}_\d:=(\{t\in C: t\d <x_\d\},\{t\in C: x_\d<t\d\})$$
is a gap of $C$ (so $x_\d\in D$).  
Let $$(A_1,B_1)<_{DM}(A_2,B_2)<_{DM}\cdots<_{DM}(A_n,B_n)$$ be the chain of distinct gaps among $\mathcal{G}_\d$ ($\d\in T$) in the Dedekind-MacNeille completion.

Let $P$ be the set of all pairs $(\a,\b)\in U\times U$ with $\a\in R$ and $z\b\notin\im\a.$  For each $(\a,\b)\in P,$ choose a right inverse $\lam_\a$ of $\a,$ and let $u_{\a,\b}=z\b\lam_\a$.  Since, by Lemma \ref{lem:CD}, the maps $\a$ and $\lam_\a$ fix $M\cup N$ pointwise and map $D$ to $D,$ and we have $z\b\notin\im\a$ by the definition of $P,$ it follows that $u_{\a,\b}=(z\b)\lam_\a\in D\lam_\a\subseteq D.$  Moreover, by Lemma \ref{lem:rightunitset} we have
$$u_{\a,\b}\a=
\begin{cases}
\max(\im\a\cap(z\b)^{\downarrow})&\text{if }u_{\a,\b}\a<z\b\\
\min(\im\a\cap(z\b)^{\uparrow})&\text{if }u_{\a,\b}\a>z\b.
\end{cases}$$
Now, applying Lemma \ref{lem:trim} to the gaps $(A_k,B_k)$ ($1\leq k\leq n$) and the finite set 
$$G=(F\cap D)\cup\{u_{\a,\b} : (\a,\b)\in P\}\:(\subseteq D),$$ 
there exist terminal segments $T_k$ of $A_k$ ($1\leq k\leq n$), initial segments $I_k$ of $B_k$ $(1\leq k\leq n),$ and intervals $D_p$ of $D$ $(1\leq p\leq n+1)$ such that 
$$D=D_1+T_1+I_1+D_2+\cdots+D_n+T_n +I_n+D_{n+1},\quad\: G\subseteq D_1+D_2+\cdots+D_{n+1},$$
and for each $p\in\{1,\dots,n+1\}$ the set $G\cap D_p$ contains no endpoints of $D_p$.  Note that each $T_k$ has no maximum, and each $I_k$ has no minimum. 

For $(\a,\b)\in P,$ with $u_{\a,\b}\in D_p$, say, we fix elements $x_{\a,\b},y_{\a,\b}\in D_p$ with $u_{\a,\b}\in(x_{\a,\b},y_{\a,\b})$; such elements exist as $u_{\a,\b}$ is not an endpoint of $D_p$.  Since $\a$ is injective, we have $u_{\a,\b}\a\in(x_{\a,\b}\a,y_{\a,\b}\a).$  It follows that $z\b\in(x_{\a,\b}\a,y_{\a,\b}\a).$  Indeed, if $u_{\a,\b}\a<z\b,$ then certainly $x_{\a,\b}\a<z\b,$ and we have $z\b<y_{\a,\b}\a$ as $u_{\a,\b}\a=\max(\im\a\cap(z\b)^{\downarrow}).$  The case in which $u_{\a,\b}\a>z\b$ is similar.

Now, choose $x\in D_{n+1}+N$ such that $x\geq F\cup\{y_{\a,\b} : (\a,\b)\in P\},$ and write $D_{n+1}+N$ as $D'+E$ where $D'=(D_{n+1}+N)\cap x^{\downarrow}$.  Observe that $E\cap F=\es.$  Moreover, if $N\neq\es$ then $x=\max(C),$ in which case $E=\es$; otherwise, $E$ is unbounded above.  
We now have
$$C=M+D_1+T_1+I_1+D_2+\cdots+D_n+T_n+I_n+D'+E.$$
Using Lemma \ref{lem:shifting}, we choose $\th\in S$ such that:
\begin{enumerate}[leftmargin=*]
\item $(M+D_1+T_1)\th\subseteq T_1$ with $\th$ injective on $M+(F\cap D_1)$ and with $T_1\th$ unbounded above in $T_1$;
\item for each $k\in\{2,\dots,n\},$ we have $(D_k+T_k)\th\subseteq T_k$ with $\th$ injective on $F\cap D_k$ and with $T_k\th$ unbounded above in $T_k$;
\item $(I_n+D')\th\subseteq I_n$ with $\th$ injective on $F\cap D'$ and with $I_n\th$ unbounded below in $I_n$;
\item $\th$ fixes $\bigcup_{j=1}^{n-1}I_j\cup E$ pointwise.
\end{enumerate}
Note that $\im\th\subseteq\bigcup_{k=1}^n(T_k+I_k)\cup E.$  Our construction of $\th$ guarantees that $\im\th\cap F=\es,$ that $\th$ is injective on $F,$ and that $\im\th$ is unbounded above in the case that $N=\es.$
Since $D_r(S;U)=2,$ there exists a $U$-sequence
$$\th=\a\g,\ \b\g=\a'\g',\ \b'\g'=1_C.$$
By the same argument as that in \textit{Case 1}, the map $\a\in U$ is injective, and if $N=\es$ then $\im\a$ is unbounded above.

Suppose first that $\a\in R,$ and note that $z\b\g=z\a'\g'\in J\subseteq F.$  If $z\b\in\im\a,$ then $z\b\g\in\im\a\g=\im\th,$ contradicting that $\im\th\cap F=\es.$  Thus $z\b\notin\im\a,$ so that $(\a,\b)\in P.$  Recalling that $x_{\a,\b},y_{\a,\b}\in D_p$ for some $p\in\{1,\dots,n+1\},$ and that $z\b\in(x_{\a,\b}\a,y_{\a,\b}\a),$ we also have
$$z\b\g\in[x_{\a,\b}\a\g,y_{\a,\b}\a\g]=[x_{\a,\b}\th,y_{\a,\b}\th]\subseteq D_p\th\subseteq
\begin{cases}
T_p&\text{if }p\leq n\\
I_n&\text{if }p=n+1.
\end{cases}$$
But then $z\b\g\in F\cap D\subseteq G$ and $z\b\g\in T_1+\cdots+T_n+I_n$, contradicting the fact that $G\subseteq D_1+\cdots+D_{n+1}$.

Finally, suppose that $\a\notin R,$ so that $\a\in T.$  Then $\mathcal{G}_\a=(A_k,B_k)$ for some $k\in\{1,\dots,n\},$ and $T_k\a<x_\a<I_k\a$.  It follows that 
$$T_k\th=T_k\a\g\leq x_\a\g\leq I_k\a\g=I_k\th.$$
Since $T_k\th\subseteq T_k$ and $I_k\th\subseteq I_k$, and $I_k$ covers $T_k$, it follows that $x_\a\g$ is either an upper bound of $T_k\th$ in $T_k$, or else a lower bound of $I_k\th$ in $I_k$.  But this contradicts the properties of $\th$ given in (1)--(4).
This completes the proof.
\end{proof}

Now, combining Propositions \ref{prop:chain,right}, \ref{prop:endpoint-shiftable}, \ref{prop:noends} and \ref{prop:notES}, we obtain Theorem \ref{thm:B}.

\end{document}